\documentclass[11pt,twoside]{article}
\usepackage{amsmath, amsfonts, amsthm, amssymb, amscd,mathrsfs}
\usepackage{bbm}

\usepackage{hyperref}
\makeatletter
\newdimen\rh@wd
\newdimen\rh@hta
\newdimen\rh@htb
\newbox\rh@box
\def\rh@measure#1{\setbox\rh@box=\hbox{$#1$}\rh@wd=\wd\rh@box \rh@hta=\ht\rh@box}

\def\widecheck#1{\rh@measure{#1}%
  \setbox\rh@box=\hbox{$\widehat{\vrule height \rh@hta width\z@ \kern\rh@wd}$}%
  \rh@htb=\ht\rh@box \advance\rh@htb\rh@hta \advance\rh@htb\p@
  \ooalign{$\vrule height \ht\rh@box width\z@ #1$\cr
           \raise\rh@htb\hbox{\scalebox{1}[-1]{\box\rh@box}}\cr}}
\makeatother

\usepackage[most]{tcolorbox} 

\newtcolorbox{markbox}{%
     enhanced, breakable, size=minimal, parbox=false, after={\par}, 
     before upper={\indent}, colback=white, 
     overlay = {\draw[line width=2pt] (frame.north east) -|
                       ([xshift=3mm]frame.east)|-(frame.south east);},
     overlay first={\draw[line width=2pt] (frame.north east) -|
                           ([xshift=3mm]frame.south east);},
     overlay middle={\draw[line width=2pt] ([xshift=3mm]frame.north east) -- 
                              ([xshift=3mm]frame.south east);},
     overlay last={\draw[line width=2pt] ([xshift=3mm]frame.north east)|-
                          (frame.south east);},
}
  
\numberwithin{equation}{section}  
\usepackage{perpage} 
\MakePerPage{footnote}  
\oddsidemargin 	.in
\evensidemargin.in
\topmargin    - 0.75in
\numberwithin{equation}{section}
        \newtheorem{theorem}{Theorem}[section]
        \newtheorem{proposition}[theorem]{Proposition}
        \newtheorem{lemma}[theorem]{Lemma}
         
        \newtheorem{definition}[theorem]{Definition}

\let\oldmarginpar\marginpar
\renewcommand\marginpar[1]{\-\oldmarginpar[\raggedleft\footnotesize #1]
{\raggedright\footnotesize #1}}

\newcommand \bei {\begin{itemize}}

\newcommand \eei {\end{itemize}}

\newcommand \be {\begin{equation}}
\newcommand \bel {\be\label}
\newcommand \ee {\end{equation}}
\newcommand \la \langle
\newcommand \ra \rangle

\newcommand	\RR 		{\mathbb R}

\newcommand \del {{\partial}}

\newcommand \eps \epsilon

\newcommand \wh {\widehat{w}} 
\begin{document}

%\title{On the semilinear wave and Klein-Gordon equations in $2+1$ dimension under null condition}

\title{Global solution to the wave and Klein-Gordon system under null condition in dimension two}

\author{Shijie Dong\footnote{Fudan University, School of Mathematical Sciences. 
%\newline
Email: dongs@ljll.math.upmc.fr, shijiedong1991@hotmail.com.
%\newline AMS classification: 35L05, 35L72, 74J30.
%{\sl Keywords.} null condition, global-in-time solutions
%}}
}}
\date{\today}

\maketitle

\begin{abstract} 
We are interested in studying the coupled wave and Klein-Gordon equations with null quadratic nonlinearities in $\RR^{2+1}$. We want to establish the small data global existence result, and in addition, we also demonstrate the pointwise asymptotic behaviour of the solution to the coupled system. The initial data are not required to have compact support, and this is achieved by applying the Alinhac's ghost weight method to both the wave and the Klein-Gordon equations.
\end{abstract}
\maketitle 

\tableofcontents

%%=====================================================================

\

\section{Introduction}
\label{sec:1}

\paragraph{Model of interest}

%Motivated by the recent work of \cite{YM0, YM1, YM2, Zha} and the previous work \cite{Georgiev, Katayama12a},
We are interested in the following coupled wave and Klein-Gordon equations
\bel{eq:model2d}
\aligned
- \Box u &=  P_1^{\alpha \beta} Q_{\alpha \beta} (u, v),
\\
- \Box v + v &=  P_2^{\alpha \beta} Q_{\alpha \beta} (u, v),
\endaligned
\ee
where 
$$
Q_{\alpha \beta} (u, v) = \del_\alpha u \del_\beta v - \del_\alpha v \del_\beta u,
\quad
Q_0 (u, v) = \del_\alpha u \del^\alpha v \text{ (to be used later)}
$$
represent the classical null forms, and $P_1^{\alpha \beta}, P_2^{\alpha \beta}$ are constants.

The prescribed initial data are denoted by 
\bel{eq:ID2d}
\big( u, \del_t u \big)(t_0, \cdot) = (u_0, u_1),
\quad
\big( v, \del_t v \big)(t_0, \cdot) = (v_0, v_1).
\ee

Our goal is to show the small data global existence result for the system \eqref{eq:model2d} (without compactness assumption on the initial data), and to study the pointwise asymptotic behaviour of the solution $(u, v)$.

Throughout of the paper, we use $A \lesssim B$ to denote $A \leq C B$ with $C$ a generic constant, and use the notation $\langle s \rangle = \sqrt{1+ |s|^2}$. The spacetime indices are represented by $\alpha, \beta, \gamma \in \{ 0, 1, 2\}$, while the space indices are denoted by $a, b, c \in \{1, 2\}$, and the Einstein summation convention is adopted unless otherwise specified. As usual, we use $L^p, W^{k, p}$ (with abbreviation $H^k = W^{k, 2}$) to denote the standard Sobolev spaces, and we might use the notation $\| \cdot \| = \| \cdot \|_{L^2(\RR^2)}$ for simple illustration.

\paragraph{Brief history}

After the seminal work on nonlinear wave and nonlinear Klein-Gordon equations in $\RR^{3+1}$, by Klainerman \cite{Klainerman86} and Christodoulou \cite{Christodoulou}, and by Klainerman \cite{Klainerman85} and Shatah \cite{Shatah}, various exciting results on nonlinear wave equations, nonlinear Klein-Gordon equations, and their coupled systems come out.
In \cite{Bachelot}, Bachelot considered a Dirac-wave-Klein-Gordon system in $\RR^{3+1}$, and then in \cite{Georgiev}, Georgiev studied the coupled wave and Klein-Gordon equations \eqref{eq:model2d} with strong null nonlinearities (i.e. nonlinearities of type $Q_{\alpha \beta}$) in dimension $\RR^{3+1}$, where the initial data are assumed to be compactly supported. The study in \cite{Bachelot, Georgiev} was generalised by \cite{Katayama12a} and many others in $\RR^{3+1}$, where more general nonlinearities were studied.

Due to the fact that the wave components and the Klein-Gordon components decay slower in low dimensions, the study of coupled wave and Klein-Gordon systems has crucial difficulties in $\RR^{2+1}$.
Recently, Ma \cite{YM1, YM2} has initialised, as far as we know, the study of coupled (quasilinear) wave and Klein-Gordon systems in $\RR^{2+1}$ using the hyperboloidal foliation method \cite{LM0, LM1, Wang}, and obtained global existence results (under the compactness assumption on the initial data) for such systems, and then extended the study to more types of (semilinear) nonlinearities,  including null forms, in \cite{YM18, YM0}. The hyperboloidal foliation method, dating back to Klainerman \cite{Klainerman85} and Hormander \cite{Hormander}, turns out to be very powerful in stduying coupled wave and Klein-Gordon systems in $\RR^{2+1}$. On the other hand, Stingo and Ifrim have also investigated the quasilinear wave and Klein-Gordon systems under the null condition using (mainly) the Fourier analysis method, and obtained the global existence result (first such result without compactness assumption) in \cite{Stingo18} and almost global existence result in \cite{Stingo19}. 
Worth to mention, there also exist many results on nonlinear wave equations in $\RR^{2+1}$, see for instance \cite{Godin, Alinhac1, Alinhac2, Cai, Zha, Yin, Yang}.

Motivated by the existing results on coupled wave and Klein-Gordon systems, our prime goal is to show the global existence result for the semilinear wave and Klein-Gordon equations under null conditions in dimension $\RR^{2+1}$, where the decay of wave and Klein-Gordon components are slower, and which will be considered to be more difficult to handle, and by relying on new techniques we do not need the compactness assumption on the initial data.
%2) we will extend the study of strong null nonlinearities to all the possible classical null nonlinearities, even though the null form $Q_0(u, v)$ is regarded to be non-consistent with the coupled system (more specifically $Q_0$ is non-consistent with the scaling vector field); 
%2) we also want to remove the restriction that the solutions have compact support by relying on new techniques.
Our argument is also expected to be applicable to more types of nonlinearities for the coupled wave and Klein-Gordon systems, which will appear in the future work.

\paragraph{Major difficulties and key ideas}

We now revisit the major difficulties arising in studying coupled wave and Klein-Gordon equations in $\RR^{2+1}$ using Klainerman's vector field method.

First, we recall that the $L^2$--type norm of the wave components cannot be bounded by the natural energy, and the following Hardy--type inequality
$$
\big\| u / |t-|x| | \big\|_{L^2(\RR^2)}
\lesssim
\big\| \del u  \big\|_{L^2(\RR^2)}
$$ 
can be used to bound the $L^2$--type norm for the wave components, see for instance \cite{Lindblad}, but the compactness assumption on the solution is required. 
In order to obtain the $L^2$--type norm bounds for the wave component without the compactness assumption, we will rely on the hidden divergence form structure of the nonlinearities $Q_{\alpha \beta} (u, v)$ (see \cite{Katayama12a}), i.e. 
$$
Q_{\alpha \beta} (u, v)
=
\del_\beta \big( \del_\alpha u v \big) - \del_\alpha \big( \del_\beta u v \big),
$$
to achieve this. And we find that one can obtain the wave decay by applying the Klainerman-Sobolev inequality in \cite{Klainerman852} (see Proposition \ref{prop:K-S}). 
%On the other hand, the Sobolev--type inequality in \cite{Georgiev}, which does require any compactness assumtion, can be used to get the Klein-Gordon decay. 
However, due to the use of the Klainerman-Sobolev inequality in Proposition \ref{prop:K-S}, an iteration procedure is expected.

Second, when treating the coupled wave and Klein-Gordon systems, the scaling vector field $L_0 = S = t\del_t + x_a \del^a$ is in general avoided to use, which is due to the fact that scaling vector field does not commute with the Klein-Gordon operator. However, we find that the conformal energy (together with other observations) of the wave component allows us to bound the $L^2$ norm of $S u$, which further allows us to treat the null form $Q_0 (u, v)$, see \eqref{eq:null000}. Worth to mention, combined with the Klainerman-Sobolev inequality in Proposition \ref{prop:K-S} we are able to get the $L^\infty$ norm of $S u$.
To be more precise, in order to estimate the null form $Q_0 (u, v)$, we rely on the estimates in Lemma \ref{lem:null} to have
$$
\langle t + |x|\rangle |Q_0(u, v)|
\lesssim
\big( | \Gamma u | +| L_0 u| \big)  \sum_{a} \big( | L_a v | + | \del v| \big),
$$
in which $\Gamma \in \{L_a, \Omega_{ab}, \del_\alpha \}$. However, the difficulty lies in estimating (the $L^2$ norm of) the term $S u$. Recall that the conformal energy for wave component in $\RR^{2+1}$ is of the form
$$
E_{con} (t, u)
=
\| S u + u \|^2_{L^2(\RR^2)} + \sum_{a<b} \| \Omega_{ab} u \|^2_{L^2(\RR^2)} + \sum_a \| L_a u \|^2_{L^2(\RR^2)}. 
$$
We still cannot get estimate on $S u$ at this stage. But thanks to the hidden divergence form structure of the null forms $Q_{\alpha\beta}$, we can first obtain the $L^2$ norm estimate of $u$, and then using the simple triangle inequality to get the $L^2$ norm estimate on $S u$. 
%Then we expect to have the $L^\infty$ estimate on $S u$ in a similar way.

Another difficulty lies in that when wave equations are coupled with the Klein-Gordon equations we might lose the $\langle t-|x|\rangle$ decay for the wave component (see the Klainerman-Sobolev inequality in Proposition \ref{prop:K-S}). However, we surprisingly find that the $\langle t-|x|\rangle$ decay can be retained by first obtaining the pointwise bound for $L_0 u$, and then by relying on the fact that
$$
\langle t- |x| \rangle |\del u|
\lesssim \big| L_0 u \big| + \big| \Gamma u \big|,
\qquad
\Gamma = \del_\alpha, L_a, \Omega_{ab}.
$$
Thus, we are allowed to gain the $\langle t- |x| \rangle$ decay for the wave components with partial derivatives $\del u$.

In addition, it is not clear how to bound the highest order energies, which is because we cannot rely on the null estimates in Lemma \ref{lem:null} in the highest order cases (due to the presence of the Klein-Gordon component $v$). Fortunately, we find that the null forms can be bounded by
$$
\big| Q_0(u, v) \big| + \big| Q_{\alpha\beta}(u, v) \big| 
\lesssim
\sum_a \big| \del_a u + x_a \del_t u/ |x| \big| \big| \del v \big|
+
\sum_a \big| \del_a v + x_a \del_t v/ |x| \big| \big| \del u \big|,
$$
and we also find the Alinhac's ghost weight method \cite{Alinhac1, Alinhac2} can be applied to Klein-Gorodn equations, and hence we are able to utilise the ghost weight energy estimates. However one more problem arises: the application of the ghost weight method demands the estimate
$$
\big| \del u \big|
\lesssim \big( 1 + \big|t-|x| \big| \big)^{-1/2 + \delta_1} (1+t)^{-1/2}
$$
to be true. In order to achieve this, on one hand, we rely on the hidden divergence form structure of the null forms $Q_{\alpha \beta} (u, v)$ again and the estimates of $\del \del u$ (see Lemma \ref{lem:ddu}) within the region $\{(t, x) : |x| \leq 2t \}$, and on the other hand we rely on the pointwise decay estimates of $L_0 u$ in the region $\{(t, x) : |x| \geq 2t \}$. More details are demonstrated in the analysis in Section \ref{sec:proof}. Worth to mention, we believe that the ghost weight method on the Klein-Gordon equations has other applications.

\paragraph{Main theorem}

We are now ready to state the main result.

\begin{theorem}\label{thm:main}[Global existence result for the coupled wave and Klein-Gordon equations]
Consider the coupled wave and Klein-Gordon system \eqref{eq:model2d}, and let $N \geq 14$ be an integer. Then for any $\delta>0$, there exits $\eps_0 >0$, such that for all $\eps < \eps_0$, and all initial data satisfying the smallness condition
\be 
\aligned
&\sum_{k \leq N+1} \Big( \big\| \langle |x| \rangle^k \nabla^k u_0 \big\|_{L^1\bigcap L^2} 
+
\big\| \langle |x| \rangle^{k+1} \nabla^k v_0 \big\|_{L^2} \Big)
\\
+
&\sum_{k\leq N}
\Big( \big\| \langle |x| \rangle^{k+1} \nabla^k u_1 \big\|_{L^1\bigcap L^2} 
+
\big\| \langle |x| \rangle^{k+2} \nabla^k v_1 \big\|_{L^2} \Big)
\leq \eps,
\endaligned
\ee
with $\nabla = (\del_a)$, the Cauchy problem \eqref{eq:model2d}--\eqref{eq:ID2d} admits a global-in-time solution $(u, v)$, which satisfies the following pointwise decay results
\bel{eq:thm-decay}
|v(t, x)| \lesssim \langle t\rangle^{-1},
\quad
|u(t, x) | \lesssim \langle t \rangle^{-1/2 + \delta},
\quad
|\del u(t, x) | \lesssim \langle t-|x| \rangle^{-3/4} \langle t \rangle^{-1/2}.
\ee

\end{theorem}

Nevertheless the slow decay nature of the wave and the Klein-Gordon components in $\RR^{2+1}$, we can still get the global-in-time solution, as well as pointwise decay results of the solution, to the system \eqref{eq:model2d} without compactness restrictions on the initial data. Together with the theorem in \cite{Katayama12a}, we know the global existence result to the system \eqref{eq:model2d} (with no compactness assumptions) is valid in all $\RR^{n+1}$, with $n\geq 2$. To the best of our knowledge, whether such result to the system \eqref{eq:model2d} holds in $\RR^{1+1}$ is still unknown. But since, as far as we know, there does not exist any (nontrivial) blow-up result on the coupled wave and Klein-Gordon systems in any dimension, so we believe the answer to the $\RR^{1+1}$ question is also positive.

\paragraph{Organisation}

The rest of the paper is planned as follows. In Section \ref{sec:pre}, we revisit some notations and some basic results on the wave and Klein-Gordon equations. Then in Section \ref{sec:linear}, we prepare some key results on estimating the $L^2$ and the $L^\infty$ estimates for the linear wave equations. Finally, we provide the proof for Theorem \ref{thm:main} by relying on the fixed point theorem in Section \ref{sec:proof}.

%====================================

\section{Preliminaries}\label{sec:pre}

\subsection{Basic notations}

In the $(2+1)$ dimensional spacetime, we adopt the signature $(-, +, +)$. We denote a point in $\RR^{2+1}$ by $(x_0, x_1, x_2) = (t, x_1, x_2)$, and denote its spacial radius by $r = \sqrt{x_1^2 + x_2^2}$.

In order to apply Klainerman's vector field method, we first introduce the vector fields:
\bei
\item Translations: $\del_\alpha$, \quad $\alpha = 0, 1, 2$.

\item Rotations: $\Omega_{ab} = x_a \del_b - x_b \del_a$,  \quad $a, b = 1, 2$.

\item Lorentz boosts: $L_a = x_a \del_t + t \del_a$, \quad $a = 1, 2$.

\item Scaling vector field: $L_0 = S = t \del_t + r \del_r$.

\eei
We will use $\Gamma$ to denote a general vector field (not the scaling vector field $L_0$) in 
$$
V := \{ \del_\alpha, \Omega_{ab}, L_a \}.
$$
In addition, we also introduce the notation of (the ghost derivative) 
$$
G_a 
:= r^{-1} \big(x_a \del_t + r \del_a \big),
$$ 
which appears in the Alinhac's ghost weight method.

Given a sufficiently nice function $w = w(t, x)$, we define its energy on the constant time slice $t = constant$ by
\be 
E_m(t, w)
:=
\int_{\RR^2} \Big(|\del_t w|^2 + \sum_a |\del_a w|^2 \Big) \, dx.
\ee
For abbreviation, we use the notation
$$
E(t, w) = E_0 (t, w).
$$

\subsection{Estimates for commutators and null forms}

The following results for commutators will be frequently used, see \cite{Sogge}.

\begin{lemma}
For any $\Gamma', \Gamma'' \in V$ we have
\be 
[\Box, \Gamma'] = 0,
\qquad
\big| [\Gamma', \Gamma''] w \big| \lesssim  \big| \Gamma w \big|,
\qquad
\big| [\Gamma, \del] w \big| + \big| [L_0, \del] w \big| \lesssim  \big| \del w \big|,
\ee
in which $w$ is sufficiently nice function.
In addition, if we act the vector field $\Gamma$ on the null forms, we further have
\be
\aligned
\big| \Gamma Q_0 (u, v) 
- Q_0 (\Gamma u, v) - Q_0 (u, \Gamma v) \big|
&= 0,
\\
\big| \Gamma Q_{\alpha \beta} (u, v) 
- Q_{\alpha \beta} (\Gamma u, v) - Q_{\alpha \beta} (u, \Gamma v) \big|
&\leq \sum_{\alpha', \beta'} \big| Q_{\alpha' \beta'} (u, v) \big|.
\endaligned
\ee
\end{lemma}

In order to estimate null forms, we need the following lemma which gives very detailed estimates on the null forms and can be found in \cite{Sogge, Yin} for example.

\begin{lemma}\label{lem:null}
It holds that
\bel{eq:null000} 
\aligned
|Q_0(u, v)|
\lesssim
& \langle t+|x|\rangle^{-1} \big(\big| L_0 u \Gamma v \big| + \big| \Gamma u \Gamma v  \big|\big),
\\
|Q_{\alpha\beta}(u, v)|
\lesssim
&\langle t+|x|\rangle^{-1} \big( \big| \Gamma v \del u \big| + \big| \Gamma u \del v \big| \big).
\\
|Q_0 (u, v)| + |Q_{\alpha\beta} (u, v) |
\lesssim
& \sum_a \big( \big| G_a u \big| |\del v| + \big| G_a v \big| |\del u| \big).
\endaligned
\ee

\end{lemma}

\subsection{Sobolev--type inequalities}

Now, in order to obtaine the pointwise wave decay or Klein-Gorodn decay estimateswe from the weighted energy bounds we recall the following inequalities. We note that the importance of the inequalities below to coupled wave and Klein-Gordon equations is that we do not need to rely on the scaling vector field $L_0 = t \del_t + x^a \del_a$.

We first revisit the Klainerman-Sobolev inequality in \cite{Klainerman852}. We note that it is not required to use the scaling vector field $L_0$ in the right hand side $L^2$--type norms, so this version is very well adapted to the study of the coupled wave and Klein-Gordon systems. However, in the inequality \eqref{eq:K-S}, we need the future information till time $2t$ when deriving the pointwise bounds for the function at time $t$, and thus we rely on the fixed point iteration method to prove Theorem \ref{thm:main}.

\begin{proposition}\label{prop:K-S}
Let $u = u(t, x)$ be a sufficiently smooth function which decays sufficiently fast at space infinity for each fixed $t \geq 0$.
Then for any $t \geq 0$, $x \in \RR^2$, we have
\bel{eq:K-S}
|u(t, x)|
\lesssim \langle t \rangle^{-1/2} \sup_{0\leq s \leq 2t, |I| \leq 3} \big\| \Gamma^I u(s) \big\|_{L^2},
\qquad
\Gamma \in V = \{ L_a, \del_\alpha, \Omega_{ab} = x^a \del_b - x^b \del_a \}.
\ee
\end{proposition}

Before recalling the following inequality, which was proved by Georgiev in \cite{Georgiev2}, we first introduce some notations. 
Denote $\{ p_j \}_0^\infty$ a usual Paley-Littlewood partition of the unity
$$
1 = \sum_{j \geq 0} p_j(s),
\qquad
s \geq 0,
$$
which also satisfies 
$$
0 \leq p_j \leq 1,
\qquad
p_j \in C_0^\infty (\RR), 
\qquad
\text{for all $j \geq 0$},
$$
as well as
$$
\text{supp } p_0 \subset (-\infty, 2],
\qquad
\text{supp } p_j \subset [2^{j-1}, 2^{j+1}],
\qquad
\text{for all $j \geq 1$}.
$$

\begin{proposition}\label{prop:G}
Let $w$ solve the Klein-Gordon equation
$$
- \Box w + w = f,
$$
with $f = f(t, x)$ a sufficiently nice function.
Then for all $t \geq 0$, it holds
\be 
\aligned
&\langle t + |x| \rangle |w(t, x)|
\\
\lesssim
&\sum_{j\geq 0,\, |I| \leq 4} \sup_{0\leq s \leq t} p_j(s) \big\| \langle s+|x| \rangle \Gamma^I f(s, x) \big\|_{L^2}
+
\sum_{j\geq 0,\, |I| \leq 4} \big\| \langle |x| \rangle p_j (|x|) \Gamma^I w(0, x) \big\|_{L^2}
\endaligned
\ee

\end{proposition}

As a consequence, we have the following simplified version of Proposition \ref{prop:G}.

\begin{proposition}\label{prop:G1}
With the same settings as Proposition \ref{prop:G}, let $\delta' > 0$ and assume 
$$
 \sum_{|I| \leq 4} \big\| \langle s+|x| \rangle \Gamma^I f(s, x) \big\|_{L^2}
\leq  C_f \langle s \rangle^{-\delta'} ,
$$
then we have
\be 
\langle t + |x| \rangle |w(t, x)| 
\lesssim
C_f
+
\sum_{|I| \leq 4} \big\| \langle |x| \rangle \Gamma^I w(0, x) \big\|_{L^2}.
\ee
\end{proposition}

\subsection{Energy estimates for wave and Klein-Gordon equations}

We first recall the conformal energy estimates for wave equations in $\RR^{2+1}$, which is rarely used but will play an important role in our analysis later.
For its proof, one refers to \cite{Alinhac-book}.

\begin{proposition}
Let $w$ be the solution to
$$
-\Box w = f,
\qquad
\big( w, \del_t w \big) (0) = (w_0, w_1),
$$
then it holds
\bel{eq:conformal-EE}
E_{con} (t, w)^{1/2}
\lesssim
E_{con} (0, w)^{1/2}
+
\int_0^t \big\| \langle t'+|x| \rangle f \big\| \, dt',
\ee
in which
\be 
E_{con} (t, w)
=
\| S w + w \|^2 + \sum_{a<b} \| \Omega_{ab} w \|^2 + \sum_a \| L_a w \|^2. 
\ee
\end{proposition}

We now extend a little bit the Alinhac's ghost weight method for wave equations, so that it can also be applied to Klein-Gordon equations.
The following ghost weight energy estimates will be frequently used, which are valid for both wave and Klein-Gordon equations.

\begin{proposition}\label{prop:gst4}
Assume $w$ is the solution to 
$$
- \Box w + m^2 w = f,
$$
then we have
\be 
\aligned
&E_{gst1, m} (t, w)
\leq
\int_{\RR^2} e^{q} \big( |\del_t w|^2 + \sum_a |\del_a w|^2 + m^2 w^2 \big) \, dx (0)
+
2 \int_0^t \int_{\RR^2}\big| f \del_t w e^q \big| \, dxdt,
\endaligned
\ee
in which 
$$
q = \int_{-\infty}^{r-t} \langle s \rangle^{-3/2} \, ds,
$$
and 
\be 
\aligned
&E_{gst1, m} (t, w)
\\=
&\int_{\RR^2} e^q \big( |\del_t w|^2 + \sum_a |\del_a w|^2 + m^2 w^2 \big) \, dx (t)
+
m^2 \int_0^t \int_{\RR^2} {e^q \over \langle r-t \rangle^{3/2}} w^2 \, dxdt
\\
+
& \sum_{a} \int_0^t \int_{\RR^2} {e^q \over \langle r-t \rangle^{3/2}} \big| G_a w \big|^2 \, dxdt.
\endaligned
\ee
\end{proposition}

\begin{proof}
The proof is almost the same as the proof for the case of $m = 0$.

We multiply on both sides of the $w$ equation with $e^q \del_t w$ to get
$$
\aligned
&{1\over 2} \del_t \big( e^q (\del w)^2 + m^2 e^q w^2  \big)
-
\del_a \big( e^q \del^a w \del_t w \big)
+
{1\over 2} {e^q \over \langle t-r \rangle^{3/2}} \sum_a \big( G_a w \big)^2
\\
+
&{m^2\over 2} {e^q \over \langle t-r \rangle^{3/2}} w^2
=
e^q f \del_t w.
\endaligned
$$
Integrating over the region $[0, t] \times \RR^2$ to arrive at the desired energy estimates. 
Hence the proof is done.

\end{proof}

Since $-\pi/2 \leq q \leq \pi/2$, we thus have the following version of the ghost weight energy estimates
\bel{eq:ghost} 
\aligned
E_{gst, m} (t, w)
\lesssim
E_m (0, w)
+
\int_0^t \int_{\RR^2} |f \del_t w| \, dxdt,
%\\
%E_{gst, m} (w, t)^{1/2}
%\lesssim
%E_m (w, 0)^{1/2}
%+
%\int_0^t \int_{\RR^2} |f | \, dxdt,
\endaligned
\ee
in which
\be 
E_{gst, m} (t, w)
=
E_m (t, w) 
+ m^2 \int_0^t \int_{\RR^2} {w^2 \over \langle r-t' \rangle^{3/2}} \, dxdt'
+ \sum_{a} \int_0^t \int_{\RR^2} { | G_a w |^2 \over \langle r-t' \rangle^{3/2}} \, dxdt'.
\ee
We note that the ghost weight energy estimates imply the usual energy estimates
\bel{eq:EE-wKG}
E_m (t, w)^{1/2}
\lesssim
E_m (0, w)^{1/2}
+
\int_0^t \|f\| \, dt.
\ee

%\be 
%\aligned
%E_{gst, m} (w, t)
%=
%&\int_{\RR^2} \big( |\del_t w|^2 + \sum_a |\del_a w|^2 + m^2 w^2 \big) \, dx (t)
%+
%m^2 \int_0^t \int_{\RR^2} {w^2 \over \langle r-t \rangle^2} \, dxdt
%\\
%+
%& \sum_{a} \int_0^t \int_{\RR^2} { | G_a w |^2 \over \langle r-t \rangle^2} \, dxdt.
%\endaligned
%\ee

Besides, we also have the following type of ghost weight energy estimates.

\begin{proposition}\label{prop:gst7}
With the same assumptions as in Proposition \ref{prop:gst4}, we get
\bel{eq:gst7}
\aligned
& m^2 \int_0^t \langle t' \rangle^{-\delta} \int_{\RR^2} {w^2 \over \langle r-t' \rangle^{3/2}} \, dxdt'
+
\sum_{a} \int_0^t \langle t' \rangle^{-\delta} \int_{\RR^2} {\big| G_a w \big|^2 \over \langle r-t' \rangle^{3/2}}  \, dxdt'
\\
\lesssim
&E_m (0, w)
+
\int_0^t \int_{\RR^2} \langle t' \rangle^{-\delta} |f \del_t w| \, dxdt'.
\endaligned
\ee
\end{proposition}

\begin{proof}
We multiply on both sides of the $w$ equation with $\langle t \rangle^{-\delta} e^q \del_t w$ to get
$$
\aligned
&{1\over 2} \del_t \big( \langle t \rangle^{-\delta} e^q (\del w)^2 + m^2 e^q w^2  \big)
+
{\delta\over 2}  \big( t \langle t \rangle^{-2-\delta} e^q (\del w)^2 + m^2 e^q w^2  \big)
-
\del_a \big( \langle t \rangle^{-\delta} e^q \del^a w \del_t w \big)
\\
+
& {1\over 2} {\langle t \rangle^{-\delta} e^q \over \langle t-r \rangle^{3/2}} \sum_a \big( G_a w \big)^2
+
{m^2\over 2} {\langle t \rangle^{-\delta} e^q \over \langle t-r \rangle^{3/2}} w^2
=
\langle t \rangle^{-\delta} e^q f \del_t w.
\endaligned
$$
We integrate over the region $[0, t] \times \RR^2$, and the facts $t \geq 0$, $1 \lesssim e^q \lesssim 1$ imply the desired energy estimates. 
We thus complete the proof.

\end{proof}

%====================================

\section{$L^2$ and $L^\infty$ estimates for wave equations}\label{sec:linear}

\subsection{$L^2$ estimates for homogeneous wave equations}

We have the following lemmas which help bound the $L^2$ norm of the solution (with no derivatives in front) to wave equations, which was used in \cite{Dong1905, Dong1910, Dong2004}.

\begin{lemma}\label{lem:linear}
Let $w$ be the solution to the linear wave equation
\bel{eq:w} 
\aligned
&- \Box w = 0,
\\
w(1, \cdot) = &w_0,
\quad
\del_t w(1, \cdot) = w_1.
\endaligned
\ee
We assume that
\be 
\| w_0 \|_{L^2} + \| w_1 \|_{L^2 \bigcap L^1}
< +\infty.
\ee
Then the following $L^2$ norm bound is valid
\be\label{eq:l2bound}
\| u \|_{L^2}
\lesssim
\| w_0 \|_{L^2}
+
\langle t\rangle^\delta \| w_1 \|_{L^2 \bigcap L^1}
\ee
for $0 < \delta \ll 1$.
\end{lemma}

\begin{proof}

Recall that the Fourier transform is defined by
$$
\wh (t, \xi) = \int_{\RR^2} w(t, x) e^{-x_a \xi^a} \, dx.
$$
We express the equation of $w$ in the Fourier space 
$$
\aligned
\del_{tt} \wh(t, \xi) + |\xi|^2 \wh(t, \xi) = 0,
\\
\wh(1, \cdot) = \wh_0,
\qquad
\del_t \wh(1, \cdot) = \wh_1.
\endaligned
$$
Then we obtain the solution $w$ in Fourier space by solving the above ordinary differential equation
$$
\wh(t, \xi)
=
\cos (t |\xi|) \wh_0
+
{\sin (t |\xi|) \over |\xi|} \wh_1.
$$
Thus we can bound the $L^2$ norm of $w$ as (recall the Plancherel's identity)
\be 
\aligned
\|w\|_{L^2}
&\lesssim
\|w_0\|_{L^2}
+
\Big\|{\sin (t |\xi|) \over |\xi|} \wh_1 \Big\|_{L^2(\RR^2)}.
\endaligned
\ee
We proceed by
$$
\Big\|{\sin (t |\xi|) \over |\xi|} \wh_1 \Big\|_{L^2(\RR^2)}
\lesssim
t^\delta \Big\|{ \wh_1 \over |\xi|^{1-\delta}} \Big\|_{L^2(\RR^2)}
\lesssim
t^\delta \Big\|{ w_1 \over \Lambda^{1-\delta}} \Big\|_{L^2},
$$
in which $\Lambda = \sqrt{-\del_a \del^a}$.
The Sobolev embedding 
$$
\Big\| {f \over \Lambda^{\delta_1}} \Big\|_{L^q}
\lesssim
\| f \|_{L^p},
\qquad
\delta_1 = {2 \over p} - {2\over q}, 
\quad
1<p<q<+\infty
$$
further implies
$$
\Big\|{\sin (t |\xi|) \over |\xi|} \wh_1 \Big\|_{L^2(\RR^2)}
\lesssim
t^\delta \| w_1 \|_{L^{2/(2-\delta)}}
\lesssim
t^\delta \|w_1 \|_{L^1 \bigcap L^2}.
$$
Gathering the estimates finishes the proof.
\end{proof}

\subsection{$L^\infty$ estimates for wave equations}

Recall that we do not have any $\langle t- |x| \rangle$ decay when applying the Klainerman-Sobolev inequality of version \eqref{eq:K-S}. But the following result helps gain $\langle t- |x| \rangle^{-1}$ decay for $\del u$ components, which is of vital importance when using the ghost weight energy estimates \eqref{eq:ghost}. Its proof can be found in \cite{Yin}.

\begin{lemma}
We have
\bel{eq:du} 
\big| \del u \big|
\lesssim
\langle t- |x| \rangle^{-1}
\big( \big| L_0 u \big| + \big| \Gamma u \big| \big),
\qquad
\big| G_a u \big|
\lesssim
\langle t+|x| \rangle^{-1} \big( \big| L_0 m \big| + \big| \Gamma m \big| \big).
\ee
\end{lemma}

Next, we recall the pointwise estimates for homogeneous waves, see for instance \cite{Yin, Shatah-book}. We note that the regularity required for the initial data is much weaker in \cite{Shatah-book}, where the Besov spaces are used, but due to some regularity loss in other places we will use the following version of estimates with proof.

\begin{lemma}\label{lem:linear2}
Let $w$ be the solution to
\bel{eq:w} 
\aligned
&- \Box w = 0,
\\
w(0, \cdot) = &w_0,
\quad
\del_t w(0, \cdot) = w_1,
\endaligned
\ee
then we have
\be 
|w|
\lesssim
\langle t \rangle^{-1/2}
\big( \| w_0 \|_{W^{2,1} \bigcap H^3 } + \| w_1 \|_{W^{1,1} \bigcap H^2} \big).
\ee

\end{lemma}

\begin{proof}
We revisit the proof given in \cite{Yin}, and since the result we need is weaker, the analysis is simpler. %, and we will only focus on the case $t \geq 1$.

The solution can be represented by $w = w^0 + w^1$, with
$$
\aligned
w^0 (t, x)
=
{1\over 2\pi} \del_t \int_{|x-y| \leq t} {w_0 (y) \, dy \over \sqrt{t^2 - |x-y|^2}},
\qquad
w^1 (t, x)
=
{1\over 2\pi} \int_{|x-y| \leq t} {w_1 (y) \, dy \over \sqrt{t^2 - |x-y|^2}}.
\endaligned
$$

We will only provide the proof for the estimate of $w^1(t, x)$ when $t \geq 2$, since other cases are either similar or simpler. We note that
$$
\aligned
\big| w^1 (t, x) \big|
\lesssim
& \Big| \int_{|p| \leq t} {w_1(x+p) \, dp \over \sqrt{t^2 - |p|^2}} \Big|
\\
\lesssim
&\langle t\rangle^{-1/2} \Big( \int_{|p| \leq t-1} {|w_1(x+p)| \, dp \over \sqrt{t - |p|}} + \Big| \int_{t-1 \leq |p| \leq t} {w_1(x+p) \, dp \over \sqrt{t - |p|}} \Big| \Big)
\\
\lesssim
&\langle  t\rangle^{-1/2} \big\| w_1 \big\|_{L^1} + \langle  t\rangle^{-1/2} \Big| \int_{t-1 \leq |p| \leq t} {w_1(x+p) \, dp \over \sqrt{t - |p|}} \Big|
\endaligned
$$
We observe that
$$
\aligned
&\Big| \int_{t-1 \leq |p| \leq t} {w_1(x+p) \, dp \over \sqrt{t - |p|}} \Big|
%\\
%\lesssim
%&\Big| \int_{S^1} \int_{t-1}^t w_1(x+\omega |p|) |p| \, d \sqrt{t - |p|} d \omega\Big|
\\
\lesssim
&\Big| \int_{S^1} \int_{t-1}^t w_1(x+\omega |p|) |p| \, d \sqrt{t - |p|} d \omega\Big|
\\
\lesssim
&\Big| \int_{S^1} w_1(x+\omega (t-1)) (t-1) \, d\omega \Big|
+
\int_{S^1} \int_{t-1}^t \big(|w_1(x+\omega |p|)| |p| + | \del w_1 (x + \omega |p|)| \big) \, d |p| d \omega
\\
\lesssim
& \int_{S^1} \int_0^{t-1 }\big( |\del w_1(x+\omega |p|)| |p| + |w_1(x+\omega |p|)| \big) \, d|p| d\omega
+
\big\| w_1 \big\|_{L^1} + \big\| \del w_1 \big\|_{L^2}
\\
\lesssim
& \int_{S^1} \int_0^{t-1 } |w_1(x+\omega |p|)| \, d|p| d\omega
+
\int_{S^1} \int_1^{t-1 } |w_1(x+\omega |p|)| |p| \, d|p| d\omega
+
\big\| w_1 \big\|_{L^1} + \big\| \del w_1 \big\|_{L^1 \bigcap L^2}
\\
\lesssim
& \big\| w_1 \big\|_{L^\infty} + \big\| w_1 \big\|_{L^1} + \big\| \del w_1 \big\|_{L^1 \bigcap L^2}
\lesssim
\big\| w_1 \big\|_{H^2} + \big\| w_1 \big\|_{L^1} + \big\| \del w_1 \big\|_{L^1 \bigcap L^2}
\endaligned
$$
We thus obtain
$$
\big| w^1 (t, x) \big|
\lesssim
\langle t \rangle^{-1/2}
\big( \big\| w_1 \big\|_{H^2} + \big\| w_1 \big\|_{L^1} + \big\| \del w_1 \big\|_{L^1 \bigcap L^2} \big),
\qquad
t \geq 2.
$$

\end{proof}

Besides, the following key observation, see for instance \cite{YM1, YM18}, claims that $\del \del u$ has extra $\langle t- |x| \rangle^{-1}$ decay than $\del u$ in the spacetime region $\{(t, x) : |x| \leq 2t \}$, and this can be used to get the $L^\infty$ bound for $\del u$ thanks to its divergence form structure.

\begin{lemma}\label{lem:ddu}
Let $w$ solve
$$
- \Box w = f,
$$
and we further assume
\be 
\big| \del w \big| + \big| \del \Gamma w \big|
\lesssim C_w \langle t\rangle^{-1/2},
\qquad
|f| 
\lesssim C_f \langle t \rangle^{-3/2},
\ee
with $C_w, C_f$ constants,
then we have
\bel{eq:ddu}
\big| \del \del w \big|
\lesssim 
\big( C_w + C_f \big)  \langle t-|x| \rangle^{-1} \langle t\rangle^{-1/2},
\qquad
\text{in } \{(t, x) : |x| \leq 2t \}. 
\ee
\end{lemma}

\begin{proof}
For completeness we revisit the proof in \cite{YM1, YM18}. Since it is easily seen that the results hold for $t \leq 1$, so we will only consider the case $t \geq 1$.

We first express the wave operator $- \Box$ by $\del_t, L_a$ to get
\be 
\aligned
-\Box
=
{(t-|x|) (t+ |x|) \over t^2} \del_{tt}
+ 2 {x^a \over t^2} \del_t L_a
- {1\over t^2} L^a L_a
+ {2 \over t} \del_t
-  {x^a \over t^2} \del_a
+ {x^a \over t^2} L_a.
\endaligned
\ee
Then we find that
$$
{1 + |t-|x||  \over t} | \del_{tt} w |
\lesssim
{1\over t} \big( \big| \del \Gamma w \big| + \big| \del w \big|  \big)
+
|f|,
$$
in which we used the relation $|x| \leq 2 t$,
and thus we are led to
$$
| \del_{tt} w |
\lesssim
\big(C_w + C_f \big)  {1 \over \langle t-|x| \rangle \langle t \rangle^{1/2}}.
$$

On the other hand, we note that the following relations hold true
$$
\aligned
\del_a \del_t 
&= - {x_a \over t} \del_t \del_t + {1\over t} \del_t L_a + {x_a \over t^2} \del_t - {1\over t^2} L_a,
\\
\del_a \del_b
&= {x_a x_b \over t^2} \del_t\del_t 
- {x_a \over t^2} \del_t L_b 
- {x_b \over t^2} \del_t L_a 
+ {1\over t} \del_b L_a 
- {\delta_{ab} \over t} \del_t
+ {x_a \over t^2} \del_b,
\endaligned
$$
which, using again $|x| \leq 2t$, means
$$
\big| \del_\alpha \del_\beta w\big|
\lesssim
\big|\del_t\del_t w \big| 
+ {1\over t} \big(\big| \del \Gamma w\big| +\big| \del w \big|  \big)
\lesssim
\big|\del_t\del_t w \big| 
+ {1\over \langle t - |x| \rangle} \big(\big| \del \Gamma w\big| +\big| \del w \big|  \big).
$$

We thus complete the proof.
\end{proof}

%====================================
%-------------------------------------------------------------------------------------------------------

\section{Proof of the main theorem}\label{sec:proof}

\subsection{Initialisation of the iteration method}

As we explained in the introduction part that the utilisation of the Klainerman-Sobolev inequality \eqref{eq:K-S} requires us to rely on an iteration procedure in order to show the global existence result for the system \eqref{eq:model2d}, we thus first provide the basics for the fixed point iteration method.

We now introduce the solution space which is denoted by $X$.

\begin{definition}\label{def:X}
Let $\phi = \phi (t, x), \psi = \psi (t, x) $ be sufficiently regular functions, and we say $(\phi, \psi)$ belongs to the function space $X$ if
\bei
\item It satisfies
\be
\big( \phi, \del_t \phi, \psi, \del_t \psi \big) (0, \cdot)
= \big(u_0, u_1, v_0, v_1 \big).
\ee

\item It satisfies
\be
\big\| (u, v) \big\|_X
\leq C_1 \eps,
\ee
in which $C_1 \gg 1$ is a large constant to be determined, the size of the initial data $\eps \ll 1$ is sufficiently small such that $C_1 \eps \ll \delta$, and the $\| \cdot \|_X$ norm is defined by
\be
\aligned
\big\| (u, v) \big\|_X
:=
&\sup_{t \geq 0,\, |I| \leq N} \langle t \rangle^{-\delta} \big( \big\| \Gamma^I u \big\| + E_{gst}(\Gamma^I u, t)^{1/2} + E_{gst, 1} (\Gamma^I v, t)^{1/2} \big) 
\\
+
&\sup_{t \geq 0,\, |I| \leq N} 
\langle t \rangle^{-\delta/2} \Big( \int_0^t \langle t' \rangle^{-\delta} \Big( \Big\| {\Gamma^I n \over \langle t'-|x| \rangle^{3/4} } \Big\|^2 
+ \Big\| {G_a \Gamma^I n \over \langle t'-|x| \rangle^{3/4} } \Big\|^2 \Big) \, dt' \Big)^{1/2}
\\
+
&\sup_{t \geq 0,\, |I| \leq N-1}  \big( E_{gst}(\Gamma^I u, t)^{1/2} + E_{gst, 1} (\Gamma^I v, t)^{1/2} \big)
\\
+
&\sup_{t\geq 0, |I| \leq N-1} \langle t \rangle \big( \big\| \Box \Gamma^I m \big\| + \big\| -\Box \Gamma^I n + \Gamma^I n \big\| \big)
+
\sup_{t \geq 0,\, |I| \leq N-2}  \langle t \rangle^{-1/2-\delta}  \big\| L_0 \Gamma^I u \big\|
\\
+
&\sup_{t \geq 0,\, |I| \leq N-6}  \langle t \rangle^{-\delta}  \big\| L_0 \Gamma^I u \big\|
+
\sup_{t \geq 0, \, |I| \leq N-5} \langle t + |x| \rangle \big| \Gamma^I v \big|
\\
+
& \sup_{t\geq 0, |I| \leq N-6} \langle t \rangle^2 \big( \big| \Box \Gamma^I m \big| + \big| -\Box \Gamma^I n + \Gamma^I n \big| \big)
\\
+
&\sup_{t \geq 0, \, |I| \leq N-9} \langle t - |x| \rangle^{3/4} \langle t \rangle^{1/2} \big| \del \Gamma^I u \big|.
\endaligned
\ee

\eei

\end{definition}

We note that the function space $X$ is a Banach space.

\subsection{The solution mapping}

\begin{definition}
Given a pair of functions $(m, n) \in X$, we define 
\be 
T(m, n) := (\phi, \psi),
\ee
in which $(\phi, \psi)$ is the solution to the following (linear) system
\bel{eq:phi-psi} 
\aligned
- \Box \phi &=  P_1^{\alpha \beta} Q_{\alpha \beta} (m, n),
\\
- \Box \psi + \psi &=  P_2^{\alpha \beta} Q_{\alpha \beta} (m, n),
\\
\big( \phi, \del_t \phi, \psi, \del_t \psi \big) (0, \cdot)
&= \big(u_0, u_1, v_0, v_1 \big).
\endaligned
\ee

\end{definition}

We have the following proposition about the solution mapping $T$.

\begin{proposition}\label{prop:mapping1}
The images of the solution mapping $T$ lie in $X$.
\end{proposition}

We need the following results to prove Proposition \ref{prop:mapping1}.

%\begin{lemma}
%Let $(m, n) \in X$, then for all $t \geq 0$ $(\phi, \psi) = T(m, n)$ satisfies
%\be 
%\aligned
%& E(\Gamma^I u, t)^{1/2} + E_1 (\Gamma^I v, t)^{1/2} 
%\lesssim \eps + (C_1 \eps)^2,
%\qquad
%&|I| \leq N-1,
%\\
%& E(\Gamma^I u, t)^{1/2} + E_1 (\Gamma^I v, t)^{1/2}
%\lesssim 
%\eps + (C_1 \eps)^2 (1+t)^\delta,
%\qquad
%&|I| \leq N.
%\endaligned
%\ee
%
%\end{lemma}

\begin{lemma}\label{lem:map00}
We have
\be 
\aligned
\big| L_0 \Gamma^I m \big|
&\lesssim C_1 \eps \langle t \rangle^{-1/2 + \delta/2},
\qquad
&|I| \leq N- 9,
\\
\big|\Gamma^I m \big|
&\lesssim C_1 \eps \langle t \rangle^{-1/2 + \delta/2},
\qquad
&|I| \leq N- 4,
\\
\big| \del \Gamma^I m \big|
&\lesssim
C_1 \eps \langle t \rangle^{-1/2},
\qquad
&|I| \leq N- 4,
\\
\big| \del \Gamma^I m \big|
&\lesssim C_1 \eps \langle t - |x| \rangle^{-1} \langle t \rangle^{-1/2 + \delta/2},
\qquad
&|I| \leq N- 9.
%\\
%\big| \Gamma^I n \big|
%&\lesssim
%C_1 \eps \langle t \rangle^{-1},
%\qquad
%&|I| \leq N- 5.
\endaligned
\ee

\end{lemma}
\begin{proof}
The first three estimates follow from the Klainerman-Sobolev inequality \eqref{eq:K-S} and the commutator estimates, and the last one is from the definition of the function space $X$.
%As for the last one, we just need to note
%$$
%\langle t-|x| \rangle \big| \del m \big|
%\lesssim
%\big|L_0 m\big| + \big| \Gamma m\big|.
%$$

\end{proof}

\begin{lemma}\label{lem:map01}
We have
\be 
\aligned
E_{gst} (\Gamma^I \psi, t)^{1/2}
&\lesssim \eps + (C_1 \eps)^{3/2},
\qquad
|I| \leq N- 1,
\\
E_{gst} (\Gamma^I \psi, t)^{1/2}
&\lesssim \eps + (C_1 \eps)^{3/2} \langle t \rangle^\delta,
\qquad
|I| \leq N,
\\
\Big( \int_0^t \langle t' \rangle^{-\delta} \Big( \Big\| {\Gamma^I \psi  \over \langle r-t \rangle^{3/2}} \Big\|^2 
+
 \Big\| {G_a \Gamma^I \psi \over \langle t'-|x| \rangle^{3/4} } \Big\|^2 \Big) \, dt' \Big)^{1/2}
&\lesssim
\eps + (C_1 \eps)^{3/2} \langle t \rangle^{\delta/2} ,
\qquad
|I| \leq N.
\endaligned
\ee

\end{lemma}

\begin{proof}

We act the vector field $\Gamma^I$ on both sides of $\psi$ equation in \eqref{eq:phi-psi} to get
$$
\aligned
- \Box \Gamma^I \psi + \Gamma^I \psi =  P_2^{\alpha \beta} \Gamma^I Q_{\alpha \beta} (m, n).
\endaligned
$$
The usual energy estimates \eqref{eq:EE-wKG} give
$$
E_1 (t, \Gamma^I \psi)^{1/2}
\lesssim
E_1 (0, \Gamma^I \psi)^{1/2}
+
\int_0^t \big\| P_2^{\alpha \beta} \Gamma^I Q_{\alpha \beta} (m, n) \big\| \, dt'.
$$

For the case $|I| \leq N-1$, we have (recall $N\geq 14$)
$$
\aligned
\big\| P_2^{\alpha \beta} \Gamma^I Q_{\alpha \beta} (m, n) \big\|
\lesssim
&\sum_{\alpha, \beta, |I_1| + |I_2| \leq |I|} \big\| Q_{\alpha \beta} (\Gamma^{I_1} m, \Gamma^{I_2} n) \big\|
\\
\lesssim 
&\sum_{|I_1| + |I_2| \leq |I|} \big\| t'^{-1} \Gamma \Gamma^{I_1} m \Gamma \Gamma^{I_2} n \big\|
\\
\lesssim
&\sum_{\substack{|I_1| + |I_2| \leq |I| \\ |I_2| \leq N-6} }  t'^{-1} \big\| \Gamma \Gamma^{I_1} m \big\| \big\| \Gamma \Gamma^{I_2} n \big\|_{L^\infty}
+
\sum_{\substack{|I_1| + |I_2| \leq |I| \\ |I_1| \leq N-5}}  t'^{-1} \big\| \Gamma \Gamma^{I_1} m \big\|_{L^\infty} \big\| \Gamma \Gamma^{I_2} n \big\|
\\
\lesssim
&(C_1 \eps)^2 t'^{-3/2 + 2 \delta}.
\endaligned
$$
So we are led to
$$
\aligned
E_1 (t, \Gamma^I \psi)^{1/2}
\lesssim
\eps
+
(C_1 \eps)^2 \int_0^t t'^{-3/2 + 2\delta} \, dt'
\lesssim
\eps + (C_1 \eps)^2.
\endaligned
$$
Then, we apply the ghost weight energy estimates \eqref{eq:ghost} to obtain
$$
E_{gst, 1} (t, \Gamma^I \psi)
\lesssim
E_{gst, 1} (0, \Gamma^I \psi)
+
\int_0^t \big\| P_2^{\alpha \beta} \Gamma^I Q_{\alpha \beta} (m, n) \del_t \Gamma^I v \big\|_{L^1} \, dt'.
$$
Similarly, we get
$$
\aligned
E_{gst, 1} (t, \Gamma^I \psi)
\lesssim
\eps^2 
+
\int_0^t \big\| P_2^{\alpha \beta} \Gamma^I Q_{\alpha \beta} (m, n) \big\| \big\| \del_t \Gamma^I v \big\| \, dt'
\\
\lesssim
\eps^2
+
(C_1 \eps)^3 \int_0^t t'^{-3/2 + 2\delta} \, dt'
\lesssim
\eps^2 + (C_1 \eps)^3.
\endaligned
$$

Next, we turn to the case of $|I| \leq N$, and we start with estimating (recall $N \geq 14$)
$$
\aligned
&\big\| P_2^{\alpha \beta} \Gamma^I Q_{\alpha \beta} (m, n) \big\|
\\
\lesssim
&\sum_{\alpha, \beta, |I_1| + |I_2| \leq |I|} \big\| Q_{\alpha \beta} (\Gamma^{I_1} m, \Gamma^{I_2} n) \big\|
\\
\lesssim 
&\sum_{a, |I_1| + |I_2| \leq |I|} 
\Big( \big\| G_a  \Gamma \Gamma^{I_1} m \del \Gamma^{I_2} n \big\| 
+ \big\| \del  \Gamma \Gamma^{I_1} m G_a \Gamma^{I_2} n \big\| \Big)
\\
\lesssim
&\sum_{\substack{|I_1| + |I_2| \leq |I| \\ |I_2| \leq N-9} }  
\Big( \Big\| {G_a \Gamma^{I_1} n \over \langle t'-|x| \rangle^{3/4} } \Big\| \big\| \langle t'-|x| \rangle^{3/4} \del \Gamma^{I_2} m \big\|_{L^\infty}
+
\big\| \del \Gamma^{I_1} n \big\| \big\| G_a \Gamma^{I_2} m \big\|_{L^\infty} \Big)
\\
+
&\sum_{\substack{|I_1| + |I_2| \leq |I| \\ |I_1| \leq N-6}}  
 \big\| \del \Gamma^{I_1} m \big\| \big\| \del \Gamma^{I_2} n \big\|_{L^\infty}.
\endaligned
$$
Recall the relation
$$
\big| G_a m \big|
\lesssim
\langle t+|x| \rangle^{-1} \big( \big| L_0 m \big| + \big| \Gamma m \big| \big),
$$
as well as the bounds
$$
\big| \langle t'-|x| \rangle^{3/4} \del \Gamma^{I_2} m \big|
\lesssim C_1 \eps \langle t \rangle^{-1/2},
$$
we thus arrive at
$$
\big\| P_2^{\alpha \beta} \Gamma^I Q_{\alpha \beta} (m, n) \big\|
\lesssim
C_1 \eps \langle t \rangle^{-1/2} \sum_{|I| \leq N} \Big\| {G_a \Gamma^I n \over \langle t-|x| \rangle^{3/4} } \Big\| 
+
(C_1 \eps)^2 \langle t \rangle^{-1 + \delta}.
$$
Then the energy estimates \eqref{eq:EE-wKG} yield
$$
\aligned
E_1 (t, \Gamma^I)^{1/2}
\lesssim
&\eps + \int_0^t \big\| P_2^{\alpha \beta} \Gamma^I Q_{\alpha \beta} (m, n) \big\| \, dt'
\\
\lesssim
&\eps 
+ \int_0^t \Big(  (C_1 \eps)^2 \langle t'\rangle^{-1+\delta} + C_1 \eps \langle t \rangle^{-1/2} \sum_{|I| \leq N} \Big\| {G_a \Gamma^I n \over \langle t-|x| \rangle^{3/4} } \Big\| \Big) \, dt'
\\
\lesssim 
&\eps + (C_1 \eps)^2 t^\delta + C_1 \eps \sum_{|I| \leq N} \Big( \int_0^t \langle t'\rangle^{-1+\delta} \, dt' \Big)^{1/2} \Big( \int_0^t \langle t' \rangle^{-\delta/2} \Big\| {G_a \Gamma^I n \over \langle t-|x| \rangle^{3/4} } \Big\|^2 \Big) \, dt' \Big)^{1/2}, 
\endaligned
$$
which leads us to
$$
E_1 (t, \Gamma^I)^{1/2}
\lesssim
\eps + (C_1 \eps)^2 \langle t \rangle^\delta.
$$
In succession, we apply the ghost weight energy estimates \eqref{eq:ghost} to get
$$
\aligned
&E_{gst, 1} (t, \Gamma^I \psi)
\\
\lesssim
&E_{gst, 1} (0, \Gamma^I \psi)
+
\int_0^t \big\| P_2^{\alpha \beta} \Gamma^I Q_{\alpha \beta} (m, n) \big\| \big\| \del_t \Gamma^I v \big\| \, dt'
\\
\lesssim
&\eps^2
+
\int_0^t \Big( (C_1 \eps)^3  \langle t'\rangle^{-1+2\delta} + (C_1 \eps)^2 \langle t' \rangle^{-1/2+\delta} \sum_{|I| \leq N} \Big\| {G_a \Gamma^I n \over \langle t-|x| \rangle^{3/4} } \Big\| \Big) \, dt'
\lesssim
\eps^2 + (C_1 \eps)^3 t^{2\delta}.
\endaligned
$$
Finally, we use the ghost weight energy estimates \eqref{eq:gst7} to proceed
$$
\aligned
& 
\sum_a\int_0^t \langle t' \rangle^{-\delta} \Big( \Big\| {\Gamma^I \psi  \over \langle r-t \rangle^{3/2}} \Big\|^2 
+
\Big\| {G_a \Gamma^I \psi \over \langle t'-|x| \rangle^{3/4} } \Big\|^2 \Big) \, dt'
\\
\lesssim
&E_m (0, \Gamma^I \psi)
+
\int_0^t \int_{\RR^2} \langle t' \rangle^{-\delta} |f \del_t \Gamma^I \psi| \, dxdt'
\\
\lesssim
& \eps^2 +  \int_0^t \Big( (C_1 \eps)^3 \langle t'\rangle^{-1+\delta} + (C_1 \eps)^2 \langle t' \rangle^{-1/2} \sum_{|I| \leq N} \Big\| {G_a \Gamma^I n \over \langle t-|x| \rangle^{3/4} } \Big\| \Big) \, dt'
\lesssim
\eps^2 + (C_1 \eps)^3 t^\delta.
\endaligned
$$
The proof is complete now.
\end{proof}

\begin{lemma}\label{lem:map02}
We have
\be 
\aligned
E_{gst} (\Gamma^I \phi, t)^{1/2}
&\lesssim \eps + (C_1 \eps)^{3/2},
\qquad
|I| \leq N- 1,
\\
E_{gst} (\Gamma^I \phi, t)^{1/2}
&\lesssim \eps + (C_1 \eps)^{3/2} \langle t \rangle^\delta,
\qquad
|I| \leq N.
\endaligned
\ee

\end{lemma}
\begin{proof}
The same proof in Lemma \ref{lem:map01} also applies here, so we omit the proof.

\end{proof}

\begin{lemma}\label{lem:map02b}
We have
\be 
\aligned
\big\| \Box \Gamma^I \phi \big\| + \big\| (-\Box +1) \Gamma^I \psi \big\| 
\lesssim
&(C_1 \eps)^2 \langle t \rangle^{-1},
\qquad
&|I| \leq N-1,
\\
\big| \Box \Gamma^I \phi \big\| + \big\| (-\Box +1) \Gamma^I \psi \big| 
\lesssim
&(C_1 \eps)^2 \langle t \rangle^{-2},
\qquad
&|I| \leq N-6.
\endaligned
\ee
\end{lemma}

\begin{proof}
The proof of the $L^2$--type norm estimates was covered in the proof of Lemma \ref{lem:map01}.

As for the sup-norm estimates for $|I| \leq N-6$, we observe that it suffices to show
$$
\big| P_a^{\alpha \beta} \Gamma^I Q_{\alpha \beta} (m, n) \big|
\lesssim (C_1 \eps)^2 \langle t \rangle^{-2},
\qquad
|I| \leq N-6.
$$
We indeed have for $|I| \leq N-6$ that
$$
\big| P_a^{\alpha \beta} \Gamma^I Q_{\alpha \beta} (m, n) \big|
\lesssim
{1\over \langle t\rangle} \sum_{|I_1|, |I_2| \leq N-5} \big| \Gamma^{I_1} m \big| \big| \Gamma^{I_2} n \big| 
\lesssim (C_1 \eps)^2 \langle t \rangle^{-5/2 + \delta}
\lesssim (C_1 \eps)^2 \langle t \rangle^{-2}.
$$
Hence we complete the proof.
\end{proof}

\begin{lemma}\label{lem:map03}
We have
\be 
\big| \Gamma^I \psi \big|
\lesssim
\eps + (C_1 \eps)^{3/2} \langle t+|x| \rangle^{-1},
\qquad
|I| \leq N-5.
\ee

\end{lemma}

\begin{proof}

According to the result in Proposition \ref{prop:G1}, it suffices to show
$$
\big\| \langle t+|x| \rangle P_2^{\alpha \beta} \Gamma^I Q_{\alpha \beta} (m, n) \big\|
\lesssim 
(C_1 \eps)^2 \langle t \rangle^{- \delta},
\qquad
|I| \leq N-1.
$$
But this was done (not exactly the same but very similar) in the proof of Lemma \ref{lem:map01}.

The proof is done.
\end{proof}

Before we proceed further, we now decompose the wave component $\phi$ as
\bel{eq:000}
\phi = \phi^5 + \del_\gamma \phi^\gamma,
\ee
in which $\phi^5, \phi^\gamma$ are solutions to the following (linear) equations:
\be 
-\Box \phi^5 = 0,
\qquad
\big( \phi^5, \del_t \phi^5 \big)(0) = (u_0, u_1),
\ee
as well as
\bel{eq:001}
-\Box \phi^\gamma = P_1^{\alpha \gamma} n \del_\alpha m - P_1^{\gamma \beta} n \del_\beta m,
\qquad
\big( \phi^\gamma, \del_t \phi^\gamma \big)(0) = (0, 0).
\ee
In addition, we reveal the hidden null structure in the equation of \eqref{eq:001} with the new variables
\be  
\Phi^\gamma := \phi^\gamma + P_1^{\alpha \gamma} n \del_\alpha m - P_1^{\gamma \beta} n \del_\beta m,
\ee
which is the solution to
\bel{eq:002}
\aligned
-\Box \Phi^\gamma
= 
&P_1^{\alpha\gamma} (-\Box n) \del_\alpha m + P_1^{\alpha \gamma} n (-\Box +1 ) \del_\alpha m - P_1^{\alpha \gamma} Q_0 ( n, \del_\alpha m )
\\
&- P_1^{\gamma\beta} (-\Box n) \del_\beta m - P_1^{\gamma \beta} n (-\Box +1 ) \del_\beta m + P_1^{\gamma \beta} Q_0 ( n, \del_\beta m ).
\endaligned
\ee

%\begin{lemma} \label{lem:map05a}
%
%
%\end{lemma}

\begin{lemma}\label{lem:map05}
We have
\bel{eq:map05}
\aligned
\big\| \Gamma^I \phi \big\|_{L^2}
&\lesssim
\Big( \eps + (C_1 \eps)^{3/2} \big) \langle t \rangle^\delta,
\qquad
&|I| \leq N.
%\\
%\big\| \Gamma^I \phi \big\|_{L^2}
%&\lesssim
%\eps + (C_1 \eps)^{3/2} t^{1/2 + \delta},
%\qquad
%&|I| \leq N.
\endaligned
\ee

\end{lemma}

\begin{proof}
First, the result in Lemma \ref{lem:linear} implies 
\be 
\| \Gamma^I \phi^5 \|
\lesssim
\eps \langle t \rangle^\delta,
\qquad
|I| \leq N.
\ee
Taking into account the relation \eqref{eq:000}, it suffices to show
\be 
\sum_\gamma E (t, \Gamma^I \phi^\gamma)^{1/2}
\lesssim \eps + (C_1 \eps)^2 \langle t \rangle^\delta,
\qquad
|I| \leq N.
\ee

Applying the usual energy estimates for the $\phi^\gamma$ equation, we get
$$
E (t, \Gamma^I \phi^\gamma)^{1/2}
\lesssim
E (0, \Gamma^I \phi^\gamma)^{1/2}
+
\int_0^t \big\| \Gamma^I \big( P_1^{\alpha \gamma} n \del_\alpha m - P_1^{\gamma \beta} n \del_\beta m \big) \big\| \, dt'.
$$
Successively, we have (recall that $N \geq 14$)
$$
\aligned
&\big\| \Gamma^I \big( P_1^{\alpha \gamma} n \del_\alpha m - P_1^{\gamma \beta} n \del_\beta m \big) \big\|
\\
\lesssim
& \sum_{|I_1| + |I_2| \leq N} \big\| \Gamma^{I_1} n \del \Gamma^{I_2} m \big\|
\\
\lesssim
& \sum_{\substack{|I_1| \leq N-5\\ |I_2| \leq N}} \big\| \Gamma^{I_1} n \big\|_{L^\infty} \big\| \del \Gamma^{I_2} m \big\| 
+
\sum_{\substack{|I_1|\leq N\\  |I_2| \leq N-9}} \Big\| {\Gamma^{I_1} n \over \langle t \rangle^{\delta/2} \langle t-|x| \rangle^{3/4} } \Big\| \big\| \langle t \rangle^{\delta/2} \langle t-|x| \rangle^{3/4} \del \Gamma^{I_2} m \big\|_{L^\infty}
\\
\lesssim
& (C_1 \eps)^2 \langle t \rangle^{-1+\delta}
+
C_1 \eps \langle t \rangle^{-1/2+\delta/2} \sum_{|I_1|\leq N} \Big\| {\Gamma^{I_1} n \over \langle t \rangle^{\delta/2} \langle t-|x| \rangle^{3/4} } \Big\|.
\endaligned
$$
Thus we have
$$
\aligned
&E (t, \Gamma^I \phi^\gamma)^{1/2}
\\
\lesssim
&\eps + (C_1 \eps)^2 \langle t \rangle^\delta
+
\sum_{|I_1|\leq N} \Big( \int_0^t \Big\| {\Gamma^{I_1} n \over \langle t' \rangle^{\delta/2} \langle t'-|x| \rangle^{3/4} } \Big\|^2 \, dt' \Big)^{1/2} \Big( \int_0^t \langle t'\rangle^{-1+\delta} \, dt' \Big)^{1/2}
\\
\lesssim
&\eps + (C_1 \eps)^2 \langle t \rangle^\delta.
\endaligned
$$

\end{proof}

\begin{lemma}\label{lem:map08}
We have
\bel{eq:map08} 
\aligned
\big\| L_0 \Gamma^I \phi \big\|_{L^2}
\lesssim
\eps + (C_1 \eps)^{3/2} t^{\delta},
\qquad
|I| \leq N - 6,
\\
\big\| L_0 \Gamma^I \phi \big\|_{L^2}
\lesssim
\eps + (C_1 \eps)^{3/2} t^{1/2 + \delta},
\qquad
|I| \leq N - 2.
\endaligned
\ee

\end{lemma}

\begin{proof}
We only provide the proof for the case $|I| \leq N-6$, and the case of $|I| \leq N-2$ can be shown in the similar way.

We apply the conformal energy estimates \eqref{eq:conformal-EE} on the equation
$$
- \Box \Gamma^I \phi =  P_1^{\alpha \beta} \Gamma^I Q_{\alpha \beta} (m, n),
$$
with $|I| \leq N-6$, to get
$$
E_{con} (t, \Gamma^I \phi)^{1/2}
\lesssim
E_{con} (0, \Gamma^I)^{1/2}
+
\int_0^t \big\| \langle t'+|x| \rangle P_1^{\alpha \beta} \Gamma^I Q_{\alpha \beta} (m, n) \big\| \, dt'.
$$
We note that
$$
\aligned
\big\| \langle t'+|x| \rangle P_1^{\alpha \beta} \Gamma^I Q_{\alpha \beta} (m, n) \big\|
\lesssim
\sum_{|I| \leq N-6} \big\| \Gamma \Gamma^I m \big\| \sum_{|I| \leq N-6} \big\| \Gamma \Gamma^I n \big\|_{L^\infty}
\lesssim
(C_1 \eps)^2 \langle t' \rangle^{-1+\delta},
\endaligned
$$
which further yields
$$
E_{con} (t, \Gamma^I \phi)^{1/2}
\lesssim
\eps + (C_1 \eps)^2 \langle t \rangle^\delta.
$$
Thus the proof is done after recalling the estimates in Lemma \ref{lem:map05}.

\end{proof}

\begin{lemma}\label{lem:map09}
We have
\be 
\big| \del \Gamma^I \phi \big|
\lesssim
\big( \eps + (C_1 \eps)^{3/2} \big) \langle t-|x| \rangle^{-3/4} \langle t \rangle^{-1/2},
\qquad
|I| \leq N-9.
\ee

\end{lemma}

\begin{proof}
It suffices to show the following two types of estimates
\bel{eq:wave001}
\big| \del \Gamma^I \phi \big|
\lesssim 
\big( \eps + (C_1 \eps)^{3/2} \big) \langle t-|x| \rangle^{-1} \langle t \rangle^{-1/2 + \delta},
\qquad
|I| \leq N-9,
\ee
and
\bel{eq:wave002}
\big| \Gamma^I \phi \big|
\lesssim 
\big( \eps + (C_1 \eps)^{3/2} \big) \langle t-|x| \rangle^{-1} \langle t \rangle^{-1/2},
\qquad
\text{in } \{ (t, x) : |x| \leq 2t \},
\qquad
|I| \leq N-9,
\ee

For the first estimate \eqref{eq:wave001}, the estimates \eqref{eq:map05}, \eqref{eq:map08}, and the commutator estimates imply
$$
\sum_{|I_1| \leq 3, |I_2| \leq N-9} 
\big(  \big\| \Gamma^{I_1} L_0 \Gamma^{I_2} \phi \big\| 
+ \big\| \Gamma^{I_1} \Gamma \Gamma^{I_2} \phi \big\| \big)
\lesssim
\big( \eps + (C_1 \eps)^{3/2} \big) \langle t \rangle^\delta.
$$ 
Then we apply the Klainerman-Sobolev inequality \eqref{eq:K-S} to get
$$
\big| L_0 \Gamma^I \phi \big| + \big| \Gamma \Gamma^I \phi \big|
\lesssim
\big( \eps + (C_1 \eps)^{3/2} \big) \langle t \rangle^{-1/2 + \delta},
\qquad
|I| \leq N-9.
$$
In junction with the fact
$$
|\del \Gamma^I \phi |
\lesssim \langle t-|x| \rangle^{-1} \big( \big| L_0 \Gamma^I \phi \big| + \big| \Gamma \Gamma^I \phi \big|  \big),
$$
we arrive at \eqref{eq:wave001}.

Next, we derive \eqref{eq:wave002}, and we only need to consider the case $|x| \leq 2t$. Recall the decomposition \eqref{eq:000} (and the commutator estimates), and we observe that it suffices to show
$$
\big| \del \Gamma^I \phi^5 \big|
+
\sum_\gamma \big| \del \del \Gamma^I \phi^\gamma \big|
\lesssim 
\big( \eps + (C_1 \eps)^{3/2} \big) \langle t-|x| \rangle^{-1} \langle t \rangle^{-1/2},
\qquad
|I| \leq N-9.
$$
Thanks to Lemma \ref{lem:linear2}, we get
$$
\big| L_0 \Gamma^I \phi^5 \big| + \big| \Gamma \Gamma^I \phi^5 \big|
\lesssim
\eps \langle t\rangle^{-1/2},
$$
and hence
$$
\big| \del \Gamma^I \phi^5 \big|
\lesssim
\eps  \langle t-|x| \rangle^{-1} \langle t \rangle^{-1/2},
\qquad
|I| \leq N-9.
$$
On the other hand, consider the definition of $\Phi^\gamma$ and the equation \eqref{eq:002}
$$
\aligned
\Phi^\gamma 
= & \phi^\gamma + P_1^{\alpha \gamma} n \del_\alpha m - P_1^{\gamma \beta} n \del_\beta m,
\\
-\Box \Phi^\gamma
= 
&P_1^{\alpha\gamma} (-\Box n) \del_\alpha m + P_1^{\alpha \gamma} n (-\Box +1 ) \del_\alpha m - P_1^{\alpha \gamma} Q_0 ( n, \del_\alpha m )
\\
&- P_1^{\gamma\beta} (-\Box n) \del_\beta m - P_1^{\gamma \beta} n (-\Box +1 ) \del_\beta m + P_1^{\gamma \beta} Q_0 ( n, \del_\beta m ),
\endaligned
$$
and the usual energy estimates easily give
$$
\sum_\gamma E (t, \Gamma^J \Phi^\gamma)^{1/2}
\lesssim
\eps + (C_1 \eps)^{3/2},
\qquad
|J| \leq N-4,
$$
which further yields
$$
\sum_\gamma \big\| \del \Gamma^J \phi^\gamma \big\|
\lesssim \eps + (C_1 \eps)^{3/2},
\qquad
|J| \leq N-4.
$$
Again, we apply the Klainerman-Sobolev inequality \eqref{eq:K-S} (and the commutator estimates) to obtain the sup-norm bounds
$$
\sum_\gamma \big| \del \Gamma^J \phi^\gamma \big|
\lesssim \big( \eps + (C_1 \eps)^{3/2} \big) \langle t\rangle^{-1},
\qquad
|J| \leq N-7.
$$
Finally, Lemma \ref{lem:ddu} implies
$$
\sum_\gamma
\big| \del \del \Gamma^I \phi^\gamma \big|
\lesssim \big( \eps + (C_1 \eps)^{3/2} \big) \langle t-|x| \rangle \langle t\rangle^{-1},
\qquad
|I| \leq N-9.
$$
Till now, the proof is complete.
\end{proof}

We are now ready to show Proposition \ref{prop:mapping1}.

\begin{proof}[Proof of Proposition \ref{prop:mapping1}]
By carefully choosing $C_1 \gg 1$ large enough and $\eps \ll 1$ sufficiently small, we get from Lemmas \ref{lem:map01}--\ref{lem:map09} that
\be 
\big\| (\phi, \psi ) \big\|_X
\leq {1\over 2} C_1 \eps,
\ee
and hence $(\phi, \psi) \in X$.

\end{proof}

\subsection{Contraction mapping and the global existence result}

We now want to show that the solution mapping $T$ is also a contraction mapping.

\begin{proposition}
$T$ is a contraction mapping from $X$ to itself, i.e.
\be 
\big\| (\phi-\phi', \psi-\psi') \big\|_X
\leq 
{1\over 2} \big\| (m-m', n-n') \big\|_X,
\ee
in which $(m, n), (m', n') \in X$, and $(\phi, \psi) = T(m, n), (\phi', \psi') = T(m', n')$.
\end{proposition}

\begin{proof}
The proof for Proposition \ref{prop:mapping1} can also be applied (we might further shrink the size of the initial data $\eps$ if needed), so we omit it.

\end{proof}

\begin{proof}[Proof of Theorem \ref{thm:main}]
By the Banach fixed point theorem, we know the mapping $T$ has a unique fixed point, which is the solution to the system \eqref{eq:model2d}
As for the pointwise decay estimates \eqref{eq:thm-decay}, they can be obtained from the definition \ref{def:X} and Lemma \ref{lem:map00}.
\end{proof}

%============================================================

%\section*{Acknowledgments} 

%\newpage

%============================================================================================

\end{document}